\documentclass[11pt,twoside]{article}
\usepackage{amsfonts}
\usepackage{amssymb}
\usepackage{amsmath}
\usepackage{epsfig}
\usepackage{bbm}
\usepackage{mathrsfs}
\usepackage{amsfonts,latexsym,bm}
\usepackage{amsmath,amsthm}
\usepackage{amssymb,amscd}
\usepackage{amsfonts,amsbsy}
\usepackage{fancyhdr,graphicx}
\usepackage[dvips]{psfrag}
\usepackage{indentfirst}
\usepackage{subfigure}
\usepackage[square, comma, sort&compress, numbers]{natbib}
\newtheorem {theorem} {Theorem}[section]
\newtheorem {proposition} [theorem]{Proposition}
\newtheorem {corollary} [theorem]{Corollary}

\theoremstyle{remark}

\theoremstyle{definition}

\newtheorem {remark} [theorem]{Remark}
\begin{document}
\setlength{\parindent}{4ex}
\setlength{\parskip}{2ex}
\setlength{\oddsidemargin}{12mm}
\setlength{\evensidemargin}{9mm}

\title
{\textsc {New lower bound for the number of critical periods for planar polynomial systems}}
\begin{figure}[b]
\rule[-0.5ex]{7cm}{0.2pt}\\
\footnotesize $^{*}$Corresponding author. E-mail address:
cenxiuli2010@163.com (X. Cen). Supported by the NSF of China (Nos. 11801582, 11771101 and 11971495) and the NSF of Guangdong Province (No. 2019A1515011239).
\end{figure}
\author
{{\textsc {Xiuli Cen}}\\[2ex]
{\footnotesize\it  School of Mathematics (Zhuhai), Sun Yat-sen University, Zhuhai, Guangdong 519082, P.R.China}}
\date{}
\maketitle {\narrower \small \noindent {\bf Abstract\,\,\,} {In this paper, we construct two classes of planar  polynomial Hamiltonian systems having a center at the origin, and
obtain the lower bounds for the number of critical periods for these systems. For polynomial potential systems of degree $n$, we provide a lower bound of $n-2$ for the number of critical periods,
and for polynomial systems of degree $n$, we acquire a lower bound of $n^2/2+n-5/2$ when $n$ is odd and $n^2/2-2$ when $n$ is even for the number of critical periods. To the best of our knowledge, these lower bounds are new, moreover the latter one is twice the existing results up to the dominant term. }

}


\section{Introduction}

The present paper is concerned with the period function of periodic orbits of planar polynomial systems.
It is one of important open problems in the study of the qualitative theory of real planar differential
systems. For a given smooth planar autonomous differential system with a continuum of periodic orbits, through a global transversal smooth section, we can parameterize them by a real number $h$. And the period function,
denoted by $T(h)$, assigns to each orbit its minimal period.

There is a rigorous study on the period function, which covers the topics including isochronicity,
see the bibliography of \cite{C} for some historic references;
monotonicity,
for example the papers \cite{CS,Z} and the references therein;
and bifurcation of critical periods.

If the period function is not monotone, the local maximum or minimum of the period function are called critical periods. It can be proved that the number of critical periods does not depend on the transversal section, or the parametrization. Studies on the number of critical periods for polynomial systems have exhibited rich results. For instance, some authors propose several criteria to bound the number
of critical periods, see the papers \cite{LL,MV,YZ,MRV} and the references therein, and some ones investigate certain specific systems with a center, and give the number
of critical periods they have, see the references \cite{CRZ,CJ,CS,CGS,FLRS,GLY,GY,GZ,G,HBHR,LZ1,LZ2,MD,RHH} and so on.

The present paper aims at providing a new lower bound for the number of critical periods for planar polynomial systems. Two classes of polynomial systems are mainly considered.\\[1ex]
\noindent $\bullet$ {\bf Polynomial potential system.} Denote $n$ the degree of the polynomial potential systems and $N^n$ the maximum number of critical periods of the period annulus around the origin. Some known results include: \\[1ex]
\indent$\bullet$ $N^2=0$, see \cite{CJ};\\[1ex]
\indent$\bullet$ $N^3=1$, see \cite{G};\\[1ex]
\indent$\bullet$ $N^4\geq2, N^5\geq3,N^{2n-1}\geq n-2$, see \cite{CJ}; \\[1ex]
\indent$\bullet$ $N^{n}\geq 2[\frac{n-2}{2}]$, see \cite{CGS} for a method similar as the second order Melnikov function method;\\[1ex]
\indent$\bullet$ $N^{5}\geq 3, N^{7}\geq 5$, see \cite{CGS} for a method similar as the high order Melnikov function method.\\[1ex]
\noindent$\bullet$ {\bf General polynomial system.}  The results on the number of critical periods for polynomial systems indicate a linear growth with the degree of the systems \cite{CJ,CGS,MD},
until Gasull, Liu and Yang \cite{GLY} give an example using reversible isochronous centers with reversible perturbations, which shows that
there exist polynomial vector fields of degree $n$ whose number of critical periods grows at least quadratically with $n$. More precisely, the number of critical periods has the expression $n^2/4+3n/2-4$ when $n$ is even and a similar one when $n$ is odd. As far as we know, this is the best result up to now.

The main results of this paper are as follows.
\begin{theorem}\label{T1}
There exist polynomial potential systems of degree $n$ whose number of critical periods is at least $n-2$.
\end{theorem}
\begin{theorem}\label{T2}
There exist polynomial systems of degree $n$ whose number of critical periods is at least $n^2/2+n-5/2$ when $n$ is odd and $n^2/2-2$ when $n$ is even.
\end{theorem}

Comparing with the existing results,  Theorem \ref{T1} shows that $N^{2k+1}\geq 2k-1$, and this result not only generalizes the results in \cite{CGS} by using the method similar as high order
Melnikov function method. At the same time, this result can be also viewed as an improvement of the result $N^{2k+1}\geq 2[\frac{2k+1-2}{2}]=2k-2$ in \cite{CGS}. Theorems \ref{T2} tells us that there exist polynomial systems of degree $n$ whose number of critical periods grows at least quadratically with $n$, and
more accurately, the number of critical periods has the expression $n^2/2+n-5/2$ when $n$ is odd and $n^2/2-2$ when $n$ is even, which is twice the result in \cite{GLY} up to the dominant term.

The idea used in this paper is totally different from the methods in other papers. To show our idea, we suppose $n=2k+1$.  We will construct a Hamiltonian system $\dot x=H_y, \dot y=-H_x$, which satisfies
that (i) $(0, 0)$ is a center and $H(0, 0)=0$; (ii) there exists $m$ positive numbers $h_1<h_2<\cdots<h_m$ so that if $h\not=h_i$ for any $i$, then $\{(x, y)|H(x,y)=h\}$ is a closed orbit and
if $h=h_{i}$ for some $i$, then $\{(x, y)|H(x,y)=h\}$ is a homoclinic orbit and a singularity (usually a cusp). Thus we can  define the period function $T(h)$, which is well defined on
each  interval $(0, h_1)$, $(h_1, h_2), \cdots, (h_m, +\infty)$ and $T(h_i)=+\infty, 1\leq i\leq m$. The figure of $T(h)$ (we let $m=4$)  can be found in Figure \ref{Fig1.sub.1}. Obviously
 there exists at least one local minimum of $T(h)$ between two critical hamiltonian values $h_i$ and $h_{i+1}$.

\begin{figure}[ht]
\centering
\subfigure[]{
\label{Fig1.sub.1}
\includegraphics[width=0.35\textwidth]{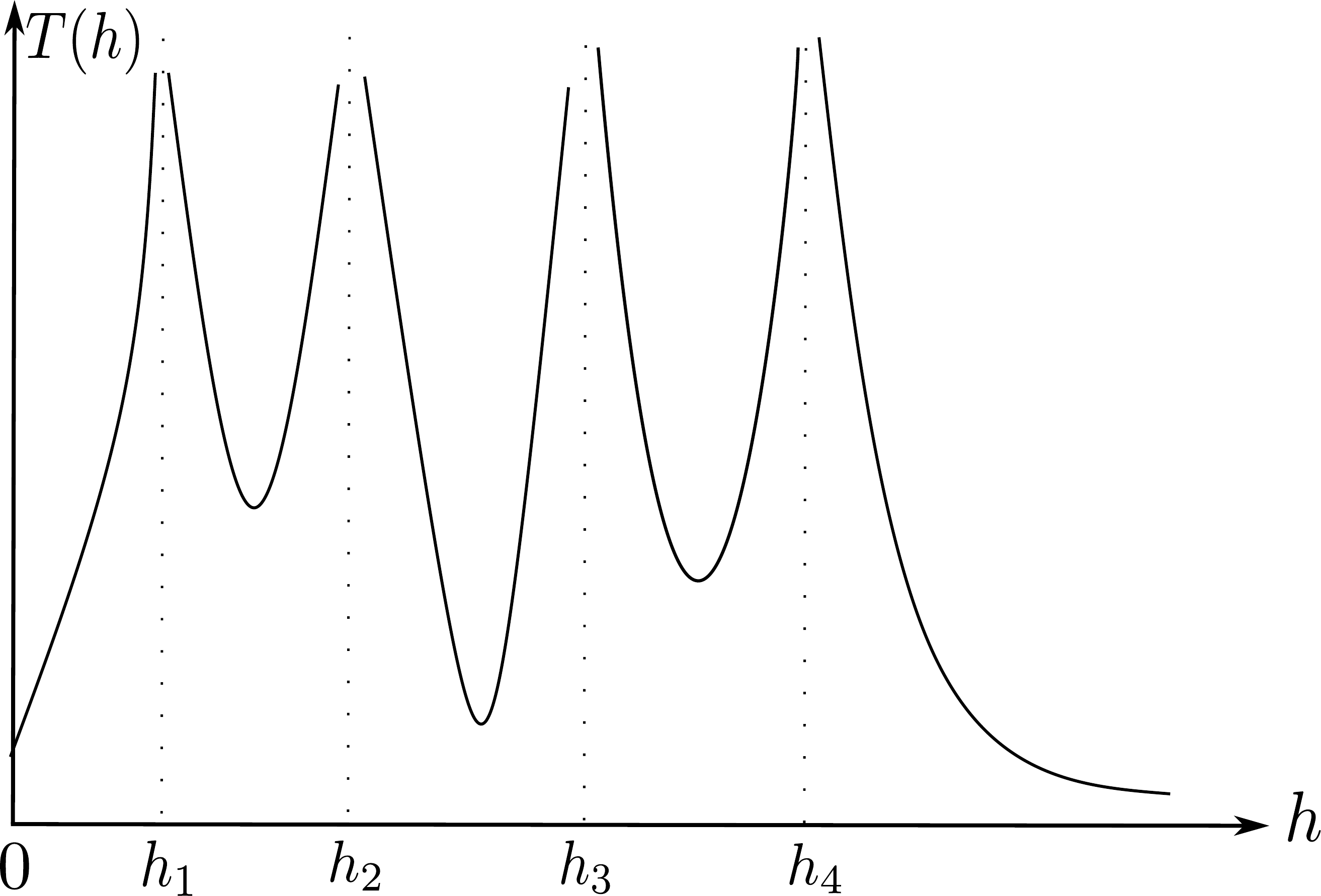}}
\subfigure[]{
\label{Fig1.sub.2}
\includegraphics[width=0.35\textwidth]{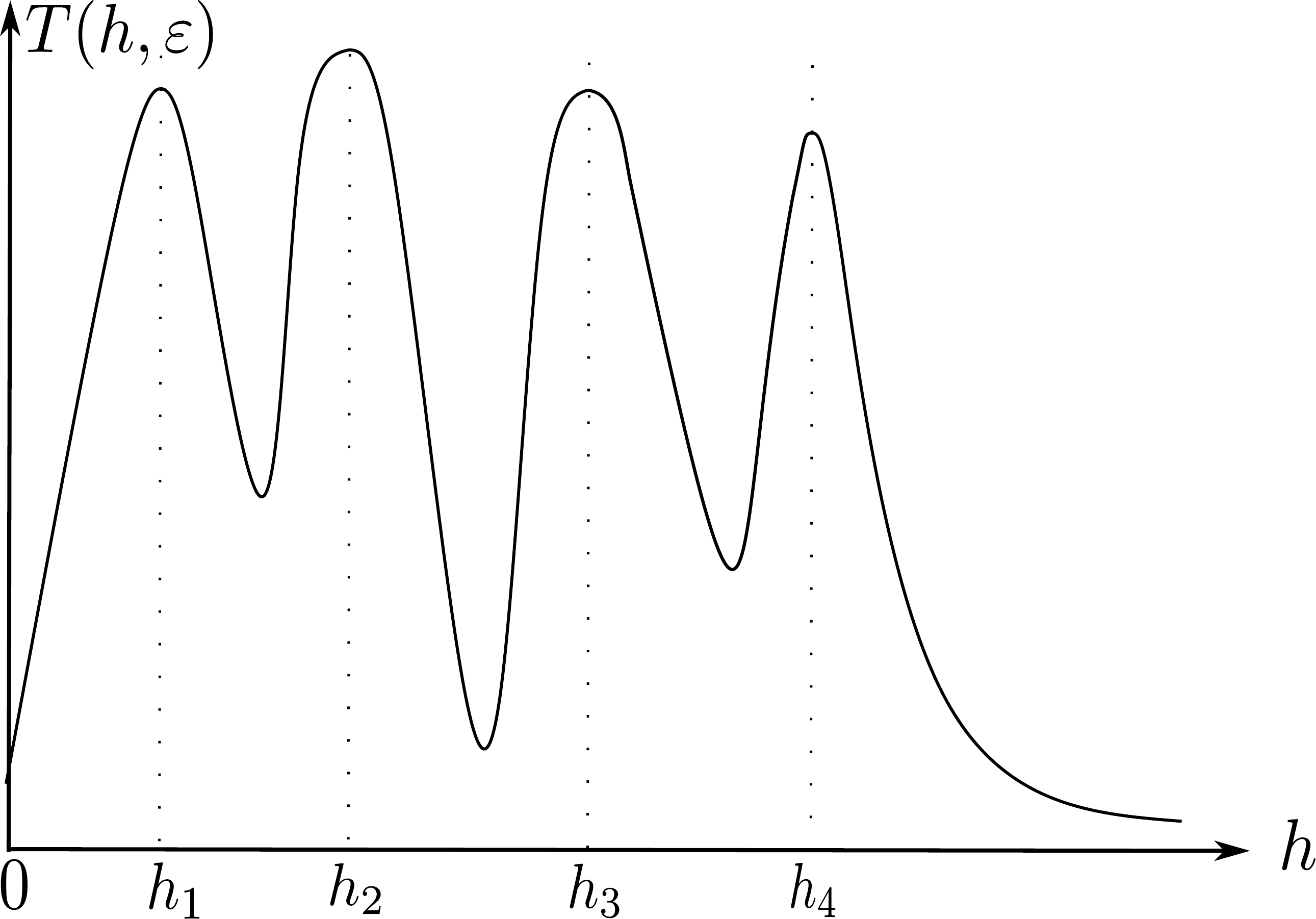}}
\caption{The graphs of the period functions $T(h)$ and $T(h,\varepsilon)$.}
\label{Fig1}
\end{figure}

Now we add some suitable perturbation so that the new system is still a Hamiltonian system but all the  cusps disappear. Then for the new system, the center is a global center so that the new period function  $T(h,\varepsilon)$ is well defined on $(0, +\infty)$. When
$\varepsilon$ is sufficiently small, the local  minimum of $T(h)$ between $h_i$ and $h_{i+1}$ will remain and there exists a  local maximum in the neighborhood of each $h_i$. See Figure \ref{Fig1.sub.2}. So $T(h,\varepsilon)$
has at least $2m-1$ critical points. Hence the problem is converted to
constructing the unperturbed Hamiltonian system such that $m$ is as big as possible, and finding suitable perturbations.

Notice that the systems we construct are Hamiltonian systems, and all the critical periods obtained occur in one nest, i.e., all the critical periodic orbits surround a unique singular point. For $n=2$, it is
well known that for Hamiltonian systems, the period functions are monotone, but for reversible systems, the period functions  can have two critical points. Thus it is natural to ask:

\noindent{\bf Question 1.}  For planar reversible polynomial systems of degree $n$, can we find a better lower bound for the number of critical periods?

There is a similar problem, which is to ask the upper bound $H(n)$ of the number of limit cycles of planar polynomial vector fields of degree $n$. Up to now, the best lower bound of $H(n)$ is
$H(n)=\frac{n^2}{2}+O(n)$ when the limit cycles are contained in one nest, see for example \cite{GGG,I,LLYZ}. But for several nests, the lower bound of  $H(n)$ becomes
$O(n^2\ln n)$, see \cite{LCC}. Notice that in our paper, the lower bound for the number of critical periods of the systems of degree $n$ is also $\frac{n^2}{2}+O(n)$. It is interesting to know (see also \cite{GLY})

\noindent{\bf Question 2.}  Are there  polynomial systems of degree $n$ whose number of critical periods occurring in several nests is at least $O(n^2\ln n)$?

The paper is organized as follows. In Section \ref{sec:PP}, the properties of the period function for unperturbed polynomial potential systems are investigated, and then a new lower bound for the number of critical periods for perturbed polynomial potential systems  is obtained. In Section \ref{sec:P}, the properties of the period function for unperturbed polynomial systems are studied, and then a new lower bound for the number of critical periods for perturbed polynomial systems is derived. Moreover, some concrete examples are given at the end.

\section{Proof of Theorem \ref{T1}}\label{sec:PP}

This section devotes to providing a lower bound for the number of critical periods for polynomial potential systems of degree $n$.
We mainly study the case $n=2k+1$, and the case $n=2k$ will be discussed briefly.

Consider the polynomial potential system
\begin{equation}\begin{split}\label{L1}
&\dot{x}=y,\\
&\dot{y}=-g(x)=-x\prod_{i=1}^k(x-\beta_i)^2,
\end{split}\end{equation}
where $\beta_i$'s are real constants, satisfying $0<\beta_1<\beta_2<\cdots<\beta_k<+\infty$.

Obviously, system \eqref{L1} is a Hamiltonian system and has a first integral
\begin{equation}\label{H0}
H(x,y)=\frac{y^2}{2}+G(x), \quad \mbox{where}\ G(x)=\int_{0}^x g(s)\ \mathrm{d}s.
\end{equation}
It is easy to verify that system \eqref{L1} has an elementary center at $(0,0)$, and $k$ cusps at $(\beta_i,0),\ i=1,2,\cdots,k$. Denote
\begin{equation}\label{hi}
h_i=H(\beta_i,0)=G(\beta_i),\ i=1,2,\cdots,k.
\end{equation}
We have that $0<h_1<h_2<\cdots<h_k<+\infty$ from $\partial H/\partial x=\mathrm{d} G/\mathrm{d} x=g(x)\geq0$ for $x>0$ and $\beta_i<\beta_{i+1}$.
System \eqref{L1} has $k+1$ period annuli around the center $(0,0)$, defined by $H(x,y)=h$ respectively on the intervals $(h_i,h_{i+1}),\ i=0,1,\cdots,k$ with $h_0=H(0,0)=0$ and $h_{k+1}=+\infty$.
Every hamiltonian value $h_i$ corresponds to a cuspidal loop passing through a unique cusp $(\beta_i,0)$. An example of the phase portrait of system \eqref{L1} is shown in Figure \ref{pp}.

\begin{figure}[ht]
\centering
\includegraphics[width=.3\textwidth]{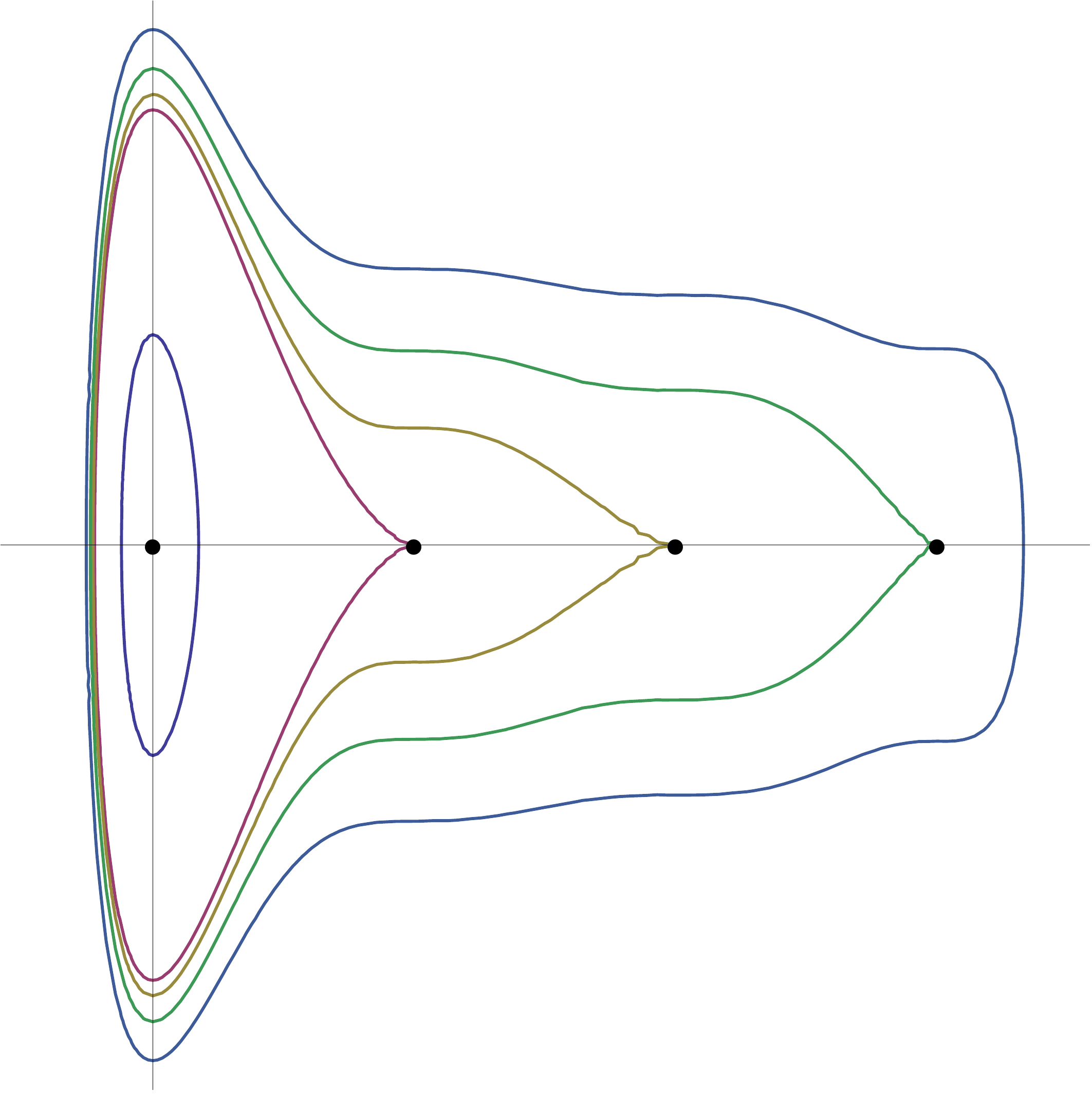}
\caption{\small{The phase portrait of system \eqref{L1} with $k=3$ and $\beta_i=i,\ i=1,2,3$.}}
\label{pp}
\end{figure}

Denote $\Gamma_h$ the periodic orbit defined by $H(x,y)=h,\ h\in\bigcup_{i=0}^k(h_i,h_{i+1})$. Then the period function $T(h)$, which assigns to $\Gamma_h$ its minimal period, can be given by
\begin{equation}\label{Th}
T(h)=\oint_{\Gamma_h}\frac{1}{y}\ \mathrm{d}x, \quad h\in\bigcup_{i=0}^k(h_i,h_{i+1}).
\end{equation}

We will study the properties of $T(h)$. Before this, two important results are introduced.

\begin{theorem}\label{le:YZ1} {\rm (\cite{YZ})}
Let $g(x)$ be a polynomial of degree $\geq$2 with real zeros, and there exists a period
annulus of system \eqref{L1} surrounding only one simple center $(x_0, 0)$, no other equilibrium.
Let $T(h)$ denote the corresponding period function. Then for a suitably chosen energy
parameter $h$, $T(h)$ is strictly monotone increasing ($T'(h)>0$)
on $(0, h_1)$, and $\lim_{h\rightarrow h_1^{-}}T(h)=+\infty$, $\lim_{h\rightarrow 0^{+}}T(h)=2\pi/\sqrt{g'(x_0)}$, where
$h_1<+\infty$.
\end{theorem}

\begin{theorem}\label{le:YZ2} {\rm (\cite{YZ})}
Let $g(x)$ be a polynomial of degree $2n+1$ ($n\geq1$) with real zeros and have a positive
leading coefficient. Then the period function $T(h)$ of the period annulus surrounding all equilibria
of system \eqref{L1} is strictly convex ($T''(h)>0$) and strictly monotone decreasing ($T'(h)<0$) on
$(\bar h_0,+\infty)$, and $\lim_{h\rightarrow +\infty}T(h)=0$, where $\bar h_0$ is finite.
\end{theorem}

Firstly, a well known fact is presented as follows.
\begin{proposition}\label{Thi}
Let $h_i$ be defined as \eqref{hi}. Then for the period function $T(h)$ given in \eqref{Th}, there hold $\lim_{h\rightarrow h^-_i}T(h)=+\infty$ and $\lim_{h\rightarrow h^+_i}T(h)=+\infty$.
\end{proposition}

Secondly, one has an essential result by Theorems \ref{le:YZ1} and \ref{le:YZ2}, which describes the properties of $T(h)$ on the intervals $(0,h_1)$ and $(h_k,+\infty)$.
\begin{proposition}\label{Th0}
Let $h_i$ be defined as \eqref{hi}. Then for the period function $T(h)$ given in \eqref{Th}, one has\\[1ex]
(i)\  $T'(h)>0$ holds on $(0,h_1)$, and $T(0)=\lim_{h\rightarrow 0^+}T(h)=2\pi/(\prod_{i=1}^k\beta_i)$, $\lim_{h\rightarrow h^-_1}T(h)=+\infty$.\\
(ii)\ $T'(h)<0$ holds on $(h_{k},+\infty)$, and $\lim_{h\rightarrow h_{k}^+}T(h)=+\infty$, $\lim_{h\rightarrow +\infty}T(h)=0$.
\end{proposition}

Since $T(h)$ is an analytic function on the interval $(h_i,h_{i+1})$,  by Proposition \ref{Thi} and Rolle's theorem, a corollary is derived as follows.
\begin{corollary}\label{coro1}
Let $h_i$ be defined as \eqref{hi}. Then for the period function $T(h)$ given in \eqref{Th}, there exists at least one local minimum point on each interval $(h_i,h_{i+1}),\ i=1,2,\cdots,k-1$.
\end{corollary}

As indicated above, we have characterized the properties of the period function $T(h)$ given in \eqref{Th}, whose graph is like as Figure \ref{Fig1.sub.1}. In the following, we will investigate a perturbed system of system \eqref{L1} and prove that there exist polynomial potential systems such that their period functions corresponding to the period annuli have at least $2k-1$ critical points on $(0,+\infty)$.

Consider the perturbed system of system \eqref{L1}
\begin{equation}\begin{split}\label{PL1}
&\dot{x}=y,\\
&\dot{y}=-g(x,\varepsilon)=-x\prod_{i=1}^k((x-\beta_i)^2+\varepsilon),
\end{split}\end{equation}
where $0<\varepsilon\ll1$ is a real parameter.

Obviously, system \eqref{PL1} is also a Hamiltonian system with Hamiltonian function
\begin{equation}\label{H1}
H(x,y,\varepsilon)=\frac{y^2}{2}+G(x,\varepsilon), \quad \mbox{where}\ G(x,\varepsilon)=\int_{0}^x g(s,\varepsilon)\ \mathrm{d}s.
\end{equation}
It is easy to verify that system \eqref{PL1} has a unique equilibrium at $(0,0)$, which is a global center.
Thus, the period annulus around the center of system \eqref{PL1} is defined by the Hamiltonian function $H(x,y,\varepsilon)=h,\ h\in(0,+\infty)$.
The period function $T(h,\varepsilon)$ corresponding to this period annulus is given by
\begin{equation}
T(h,\varepsilon)=\oint_{\Gamma_{h,\varepsilon}}\frac{1}{y}\ \mathrm{d}x,\quad h\in(0,+\infty),
\end{equation}
where $\Gamma_{h,\varepsilon}$ represents the period orbit defined by $H(x,y,\varepsilon)=h$.

Finally, a result on the lower bound for the number of critical periods for polynomial potential systems is obtained as follows.

\begin{theorem}\label{Th1}
For the period annulus of system \eqref{PL1} with $\varepsilon$ sufficiently small, the corresponding period function $T(h,\varepsilon)$ has at least $2k-1$ critical points on $(0,+\infty)$.
\end{theorem}
\begin{proof}
From Corollary \ref{coro1}, we can suppose that the period function $T(h)$ has a critical point at $h=\bar h_i\in(h_i,h_{i+1}),\ i=1,2,\cdots,k-1$.
Note that each $T(\bar h_i)$ is finite, thus we can let $M=\max\{1+T(0), 1+T(\bar h_i),\ i=1,2\cdots,k-1\}$.

Since  $\lim_{h\rightarrow h^-_i}T(h)=+\infty$, there exists $\varepsilon_i>0$ such that
$T(h_i,\varepsilon)>M$ for $0<\varepsilon<\varepsilon_i$. Furthermore, since $T(\bar h_i,\varepsilon)$ is continuous with respect to $\varepsilon$ and $T(\bar h_i,0)=T(\bar h_i)$, there exists $\bar\varepsilon_i>0$ such that $T(\bar h_i,\varepsilon)<M$ for $0<\varepsilon<\bar \varepsilon_i,\ i=1,2,\cdots,k-1$. Similarly, there exist $\bar\varepsilon_0>0$ and $\bar\varepsilon_k>0$ such that $\lim_{h\rightarrow0^+}T(h,\varepsilon)<M$ for $0<\varepsilon<\bar \varepsilon_0$ and
$\lim_{h\rightarrow+\infty}T(h,\varepsilon)<M$ for $0<\varepsilon<\bar \varepsilon_k$.

Let $\varepsilon_0=\min\{\varepsilon_1,\varepsilon_2,\cdots,\varepsilon_k, \bar\varepsilon_0,\bar\varepsilon_1,\cdots,\bar\varepsilon_{k}\}$, then
when $0<\varepsilon<\varepsilon_0$,  $T(h_i, \varepsilon)>M, i=1, 2, \cdots, k$ and $T(\bar h_i, \varepsilon)<M, i=1, 2, \cdots, k-1$. Notice that
$T(h, \varepsilon)$ has a maximum point, defined as $\hat h_i$, on any interval $[\bar h_i, \bar h_{i+1}], i=1, 2, \cdots, k-2$. Obviously
$\hat h_i\not=\bar h_i, \bar h_{i+1}$, thus $\hat h_i$'s are different. Similarly, there are at least one  local maximum point on
$(0, \bar h_1)$ and $(\bar h_{k-1}, +\infty)$, then we obtain $k$ different local maximum points. Similar discussion
shows that there are at least one  local minimum point on each interval
$(h_{i}, h_{i+1}), i=1, 2, \cdots, k-1$, then we obtain $k-1$ different local minimum points. Thus, we have Theorem \ref{Th1}.
\end{proof}

\begin{theorem}\label{Th1e}
Consider polynomial potential system of degree $2k$ with the form
\begin{equation}\begin{split}\label{Pe}
&\dot{x}=y,\\
&\dot{y}=-x\prod_{i=1}^{k-1}((x-\beta_i)^2+\varepsilon)(\beta_k-x),
\end{split}\end{equation}
where $0<\beta_1<\beta_2<\cdots<\beta_k$ are real constants and $\varepsilon>0$ is a sufficiently small parameter.
The period function corresponding to its period annulus has at least $2k-2$ critical points on $(0,\bar h_k)$, where
$\bar h_k=H(\beta_k,0,\varepsilon)$ with $H(x,y,\varepsilon)$ being a Hamiltonian first integral of system \eqref{Pe}.
\end{theorem}
\begin{proof}
It is easy to verify that the unperturbed system \eqref{Pe}$|_{\varepsilon=0}$ has an elementary center at $(0,0)$, $k-1$ cusps at $(\beta_i,0), i=1,2,\cdots,k-1$ and a saddle at $(\beta_k,0)$.
Let $h_i=H(\beta_i,0,0), i=1,2,\cdots,k-1$, then we have similar conclusions as Proposition \ref{Thi} and Corollary \ref{coro1}. The perturbed system \eqref{Pe} has a center at $(0,0)$ and the period
annulus around this center is defined by $H(x,y,\varepsilon)=h, h\in(0,\bar h_k)$. Notice that the separatrix polycycle surrounding the period annulus is a saddle loop defined by $H(x,y,\varepsilon)=\bar h_k$.
Thus the period function $T(h,\varepsilon)$ corresponding to the period annulus of system \eqref{Pe} satisfies $\lim_{h\rightarrow \bar h_k^-} T(h,\varepsilon)=+\infty$. Theorem \ref{Th1e} follows from a similar proof as Theorem \ref{Th1}.
\end{proof}
\begin{remark}
In fact, a more general condition that $H(\beta_i,0)\neq0$ differs from each others than the condition $0<\beta_1<\beta_2<\cdots<\beta_k<+\infty$ still can guarantee that
Theorem \ref{Th1} holds. The latter one is just a special case of the former one. For Theorem \ref{Th1e}, a more general condition requires that $H(\beta_i,0)\neq0, i=1,2,\cdots,k-1$ differs from each others
and $|H(\beta_k,0)|>\max\{|H(\beta_i,0)|, i=1,2,\cdots,k-1\}$.
\end{remark}

\section{Proof of Theorem \ref{T2}}\label{sec:P}

In this section we investigate a lower bound for the number of critical periods for polynomial systems of degree $n$ and prove Theorem \ref{T2}.
We mainly study the case $n=2k+1$, and the case $n=2k$ will be discussed briefly.

Consider the polynomial differential system
\begin{equation}\begin{split}\label{P2}
&\dot{x}=f(y)=y\prod_{i=1}^k(y-\alpha_i)^2,\\
&\dot{y}=-g(x)=-x\prod_{i=1}^k(x-\beta_i)^2,
\end{split}\end{equation}
where $\alpha_i$'s and $\beta_i$'s are real constants, satisfying
\begin{equation*}\begin{split}
(\mathrm{H}):\ & h_{ij}=H(\beta_i,\alpha_j),\ i,j=0,1,\cdots,k \mbox{ differs from each others.}
\end{split}\end{equation*}
Here we denote $\alpha_0=\beta_0=0$ and $H(x,y)$ is a first integral of system \eqref{P2} with the form
\begin{equation}\label{H3}
H(x,y)=F(y)+G(x), \ \mbox{where}\ F(y)=\int_{0}^y f(s)\ \mathrm{d}s\ \mbox{and}\ G(x)=\int_{0}^x g(s)\ \mathrm{d}s.
\end{equation}

By the hypothesis $(\mathrm{H})$, it is easy to verify that $\alpha_i\neq\alpha_j\neq0$ and $\beta_i\neq\beta_j\neq0$ hold for all $i\neq j,\ i,j=1,2,\cdots,k$.
It follows that system \eqref{P2} has totally $(k+1)^2$ singularities, where $(0,0)$ is an elementary center, $(\beta_i,0)$ and $(0, \alpha_i)$ are cusps, and others with the forms $(\beta_i, \alpha_j)$ are degenerate singularities with $i,j=1,2,\cdots,k$.
Using the technique of blow up, the local phase portrait of system \eqref{P2} at the degenerate singularities $(\beta_i, \alpha_j)$ is as follows, see Figure \ref{pp1}.
\begin{figure}[ht]
\centering
\includegraphics[width=.4\textwidth]{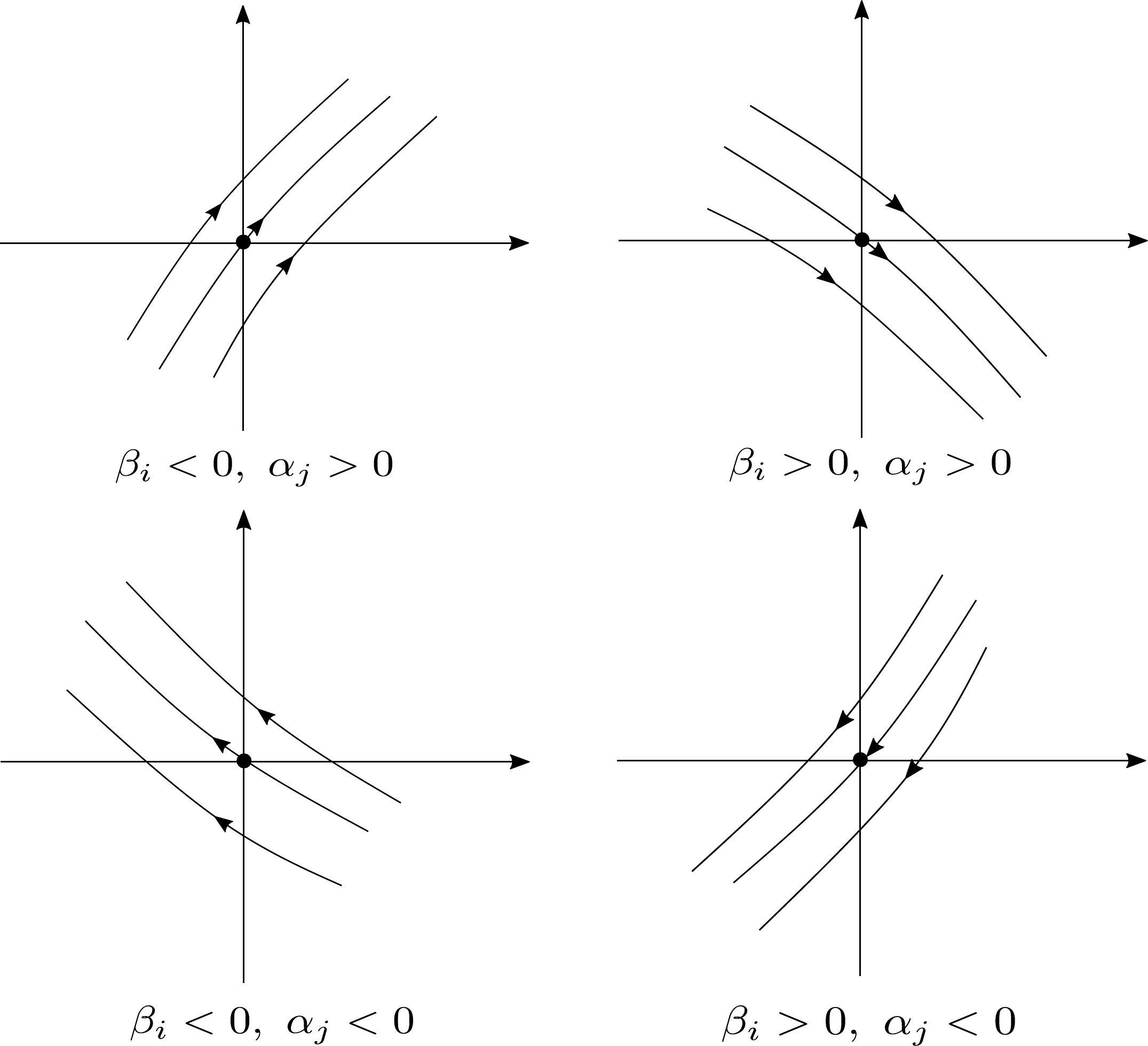}
\caption{\small{The local phase portrait of system \eqref{P2} at the degenerate singularities $(\beta_i, \alpha_j)$.}}
\label{pp1}
\end{figure}

Depending on the size of $h_{ij}$, we let
\begin{equation*}
(\mathrm{H'}): h_{s}=\{ \mbox{the $(s+1)$-th smallest element of $h_{ij}$ defined in } (\mathrm{H})\},\ s=0,1,\cdots,(k+1)^2-1.
\end{equation*}
That is, $h_0<h_1<\cdots<h_{(k+1)^2-1}$. Moreover, by the fact that the center $(0,0)$ is the minimum point of $H(x,y)$, we conclude that $h_0=H(0,0)=0$. Thus system \eqref{P2} has $(k+1)^2$ period annuli around the center, defined by $H(x,y)=h$ respectively on the intervals $(h_s,h_{s+1}),\ s=0,1,\cdots,(k+1)^2-1$ with $h_{(k+1)^2}=+\infty$.
Every Hamiltonian value $h_s,\ s=1,2,\cdots,(k+1)^2-1$ corresponds to a singular closed orbit passing through a unique singularity. An example of the phase portrait of system \eqref{P2} is shown in Figure \ref{pp2}.

\begin{figure}[ht]
\centering
\includegraphics[width=.35\textwidth]{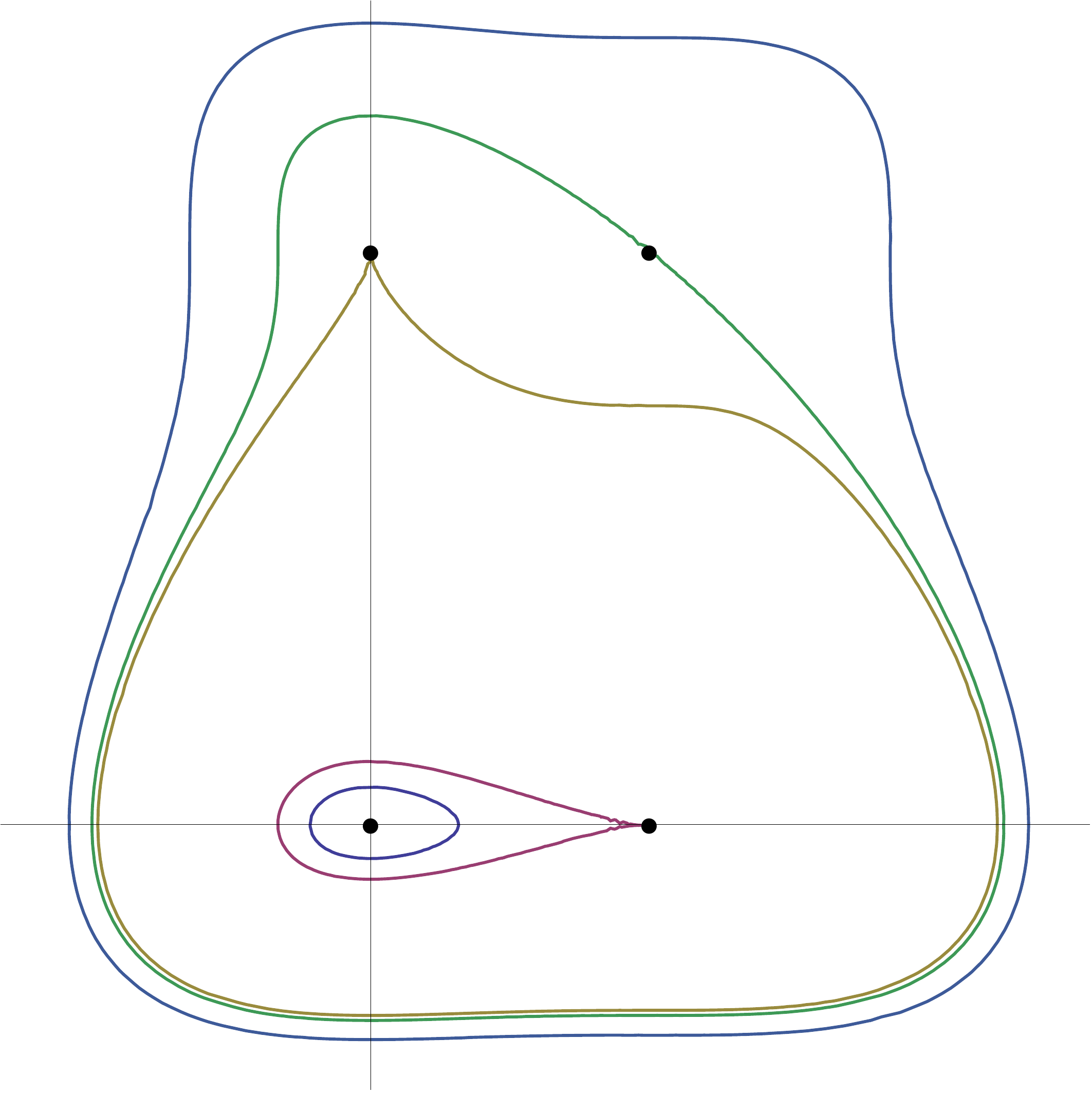}
\caption{\small{The phase portrait of system \eqref{P2} with $k=1$, $\alpha_1=4$ and $\beta_1=2$.}}
\label{pp2}
\end{figure}

Denote $\Gamma_h$ the periodic orbit defined by $H(x,y)=h,\ h\in\bigcup_{s=0}^{(k+1)^2-1}(h_s,h_{s+1})$, see \eqref{H3}. Then the period function $T(h)$, which assigns to $\Gamma_h$ its minimal period, can be given by
\begin{equation}\label{GT}
T(h)=\oint_{\Gamma_h}\frac{1}{f(y)}\ \mathrm{d}x, \qquad h\in\bigcup_{s=0}^{(k+1)^2-1}(h_s,h_{s+1}).
\end{equation}

Firstly, some properties of the period function $T(h)$ will be studied and presented.
\begin{proposition}\label{GThi}
Let $h_s$ be defined as the hypothesis $(\mathrm{H'})$.  Then for the period function $T(h)$ given in \eqref{GT}, there hold $\lim_{h\rightarrow h^-_s}T(h)=+\infty$ and $\lim_{h\rightarrow h^+_s}T(h)=+\infty$, $s=1,2,\cdots,(k+1)^2-1$.
\end{proposition}

\begin{proposition}\label{Ppro}
Let $h_s$ be defined as the hypothesis $(\mathrm{H'})$. Then for the period function $T(h)$ given in \eqref{GT}, we have that $T(0)=\lim_{h\rightarrow 0^+}T(h)=2\pi/(\prod_{i=1}^k|\alpha_i\beta_i|)$ and $\lim_{h\rightarrow +\infty}T(h)=0$.
\end{proposition}
\begin{proof}
Firstly, using the transformation $G(x)=h\cos^2\theta, F(y)=h\sin^2\theta, x\cos\theta\geq0, y\sin\theta\geq0$, $h\in(0,h_1)$, the period function defined by \eqref{GT} becomes
\begin{equation}\label{TTT}
T(h)=\int_{0}^{2\pi}\frac{2h\sin\theta\cos\theta}{f(y(\theta,h))g(x(\theta,h))}\mathrm{d}\theta,
\end{equation}
where $x(\theta,h)$ and $y(\theta,h)$ are implicit functions uniquely determined by the transformation above and they have the asymptotic expressions
\begin{equation*}\begin{split}
x(\theta,h)&=\frac{\sqrt{2h}}{\prod_{i=1}^k|\beta_i|}\cos\theta+O(h\cos^2\theta),\\
y(\theta,h)&=\frac{\sqrt{2h}}{\prod_{i=1}^k|\alpha_i|}\sin\theta+O(h\sin^2\theta).
\end{split}\end{equation*}
Thus the period function \eqref{TTT} has the expression
\begin{equation*}
T(h)=\int_{0}^{2\pi}\frac{2h\sin\theta\cos\theta}{2h\prod_{i=1}^k|\alpha_i\beta_i|\sin\theta\cos\theta+o(h\sin\theta\cos\theta)}\mathrm{d}\theta.
\end{equation*}
The first conclusion follows by taking limit for $T(h)$ when $h\rightarrow0^+$.

In the following, we will prove the second conclusion using the method in \cite{YZ}. Consider the transformation $h=G(A), A>x_2(h_s)>0$ sufficiently large, where $x_2(h_s)$ is the abscissa of the intersection point of the closed orbit defined by $H(x,y)=h_s$ with the positive $x$-axis, the first integral \eqref{H3} becomes
$F(y)+G(x)=G(A)$, and the period function $T(h)$ becomes
\begin{equation}\begin{split}\label{TA}
T(A)=&\ \int_{0}^A\frac{\mathrm{d}x}{f(y_+(x,A))}+\int_{A}^0\frac{\mathrm{d}x}{f(y_-(x,A))}+\int_{0}^{-B}\frac{\mathrm{d}x}{f(y_-(x,-B))}+\int_{-B}^0\frac{\mathrm{d}x}{f(y_+(x,-B))}\\
=&\ \int_{0}^A\left(\frac{1}{f(y_+(x,A))}-\frac{1}{f(y_-(x,A))}\right)\mathrm{d}x+\int_{-B}^0\left(\frac{1}{f(y_+(x,-B))}-\frac1{f(y_-(x,-B))}\right)\mathrm{d}x\\
=&\ AJ(A)+BJ(-B),
\end{split}\end{equation}
where $B$ satisfies $G(-B)=G(A)$, $y_+(x,R)$ and $y_-(x,R)$ represent the positive solution and the negative solution of $F(y)+G(x)=G(R)$ respectively, and
\[
J(R)=\int_{0}^1\left(\frac{1}{f(y_+(Rz,R))}-\frac{1}{f(y_-(Rz,R))}\right)\mathrm{d}z.
\]
Note that
\begin{equation}\begin{split}
G(R)-G(Rz)=&\sum_{i=1}^{2k+1}\frac{b_iR^{i+1}(1-z^{i+1})}{i+1}\\
=&\frac{R^{2k+2}}{2k+2}(1-z^{2k+2})\left(1+\sum_{i=1}^{2k}\frac{(2k+2)b_i}{i+1}e_i(z)R^{i-2k-1}\right),
\end{split}\end{equation}
where $b_{i}, i=1,2\cdots,2k+1$ are the coefficients of $g(x)=x\prod_{i=1}^k(x-\beta_i)^2=\sum_{i=1}^{2k+1}b_ix^i$ with $b_{2k+1}=1$, and
\[
e_i(z)=\frac{1-z^{i+1}}{1-z^{2k+2}}\in(0,1],\quad \mbox{for}\ z\in[0,1),\ 1\leq i\leq 2k.
\]
It is easy to see that for $\varepsilon>0$ sufficiently small, there eixsts $|R|>(2k+2)M/\varepsilon+1$ such that
\[
\left|\sum_{i=1}^{2k}\frac{(2k+2)b_i}{(i+1)}e_i(z)R^{i-2k-1}\right|\leq(2k+2)M\left(\frac{1}{|R|}+\frac{1}{|R|^2}
+\cdots+\frac{1}{|R|^{2k}}\right)\leq\frac{(2k+2)M}{|R|-1}<\varepsilon,
\]
where $M=\max\{|b_i|,i=1,2,\cdots,2k\}$. Notice that the leading coefficient of the function $F(y)$ is also $1/(2k+2)$, thus we have the expansions of $y_{+}(Rz,R)$ and $y_{-}(Rz,R)$ as follows
\begin{equation}\begin{split}\label{y+y-}
y_{+}(Rz,R)=&\ |R|(1-z^{2k+2})^{\frac{1}{2k+2}}\left(1+\frac{f_1(z)}{R}+\frac{f_2(z)}{R^2}+\cdots\right),\\
y_{-}(Rz,R)=&-|R|(1-z^{2k+2})^{\frac{1}{2k+2}}\left(1+\frac{g_1(z)}{R}+\frac{g_2(z)}{R^2}+\cdots\right).
\end{split}\end{equation}
It follows from the uniform convergence of the series \eqref{y+y-} for $z\in[0,1]$ and $|R|>(2k+2)M/\varepsilon+1$ that the expansion of the integrand of $J(R)$ is uniformly convergent, and
\begin{equation*}\begin{split}
J(R)=&\int_{0}^1\frac{f(y_-(Rz,R))-f(y_+(Rz,R))}{f(y_+(Rz,R))f(y_-(Rz,R))}\mathrm{d}z\\
=&\int_{0}^1\frac{(y_--y_+)(\prod_{i=1}^k\alpha_i^2+\cdots+y_-^{2k}+\cdots+y_+^{2k})}{\prod_{i=1}^k\alpha_i^4y_+y_-+\cdots+y_+^{2k+1}y_-^{2k+1}}\mathrm{d}z\\
=&\int_{0}^1\frac{-2|R|(1\!-\!z^{2k+2})^{\frac{1}{2k+2}}\left(1\!+\!\frac{u_1(z)}{2R}\!+\!\cdots\right)
\left((2k\!+\!1)R^{2k}(1\!-\!z^{2k+2})^{\frac{2k}{2k+2}}+\cdots\right)}
{-\prod\limits_{i=1}^k\alpha_i^4R^2(1\!-\!z^{2k+2})^{\frac{2}{2k+2}}\left(1\!+\!\frac{u_1(z)}{R}\!+\!\cdots\right)\!+\!\cdots-R^{4k+2}(1\!-\!z^{2k+2})^{\frac{2k+1}{k+1}}(1\!+\!\cdots)}\mathrm{d}z\\
=&\int_{0}^1\frac{2(2k+1)}{|R|^{2k+1}(1-z^{2k+2})^{\frac{2k+1}{2k+2}}}\left(1+\frac{v_1(z)}{R}+\frac{v_2(z)}{R^2}+\cdots\right)\mathrm{d}z\\
=&\frac{2(2k+1)}{|R|^{2k+1}}\left(a_0+\frac{a_1}{R}+\cdots\right),
\end{split}\end{equation*}
where $u_1(z)=f_1(z)+g_1(z)$,
\[
a_0=\int_{0}^1\frac{1}{(1-z^{2k+2})^{\frac{2k+1}{2k+2}}}\mathrm{d}z, \quad a_1=\int_{0}^1\frac{v_1(z)}{(1-z^{2k+2})^{\frac{2k+1}{2k+2}}}\mathrm{d}z.
\]
Hence the period function \eqref{TA} has the expansion
\[
T(A)=AJ(A)+BJ(-B)=2(2k+1)\left(\left(\frac{1}{A^{2k}}+\frac{1}{B^{2k}}\right)a_0+\left(\frac{1}{A^{2k+1}}-\frac{1}{B^{2k+1}}\right)a_1+\cdots\right).
\]
By Lemma 5.2 of \cite{YZ}, we know that
\[
B=B(A)=A(1+\frac{c_1}{A}+\frac{c_2}{A^2}+\cdots),\quad A\gg 1.
\]
Further, we have
\[
T(A)=\frac{4(2k+1)}{A^{2k}}\left(a_0-\frac{ka_0c_1}{A}+\cdots\right).
\]
The second conclusion holds since $\lim_{h\rightarrow+\infty}T(h)=\lim_{A\rightarrow+\infty}T(A)=0$.
\end{proof}

Since $T(h)$ is an analytic function on the interval $(h_s,h_{s+1})$,  by Proposition \ref{GThi} and Rolle's theorem, we have a corollary as follows.
\begin{corollary}\label{Pcoro}
Let $h_s$ be defined as the hypothesis $(\mathrm{H'})$. Then for the period function $T(h)$ given in \eqref{GT}, there exists at least one critical point on each interval $(h_s,h_{s+1}),\ s=1,2,\cdots,(k+1)^2-2$.
\end{corollary}

In the following, we will study a perturbed system of system \eqref{P2} and show that
there exist polynomial systems such that their period functions corresponding to
the period annuli have at least $2k^2+4k-1=n^2/2+n-5/2$ critical points on $(0,+\infty)$.

Consider the perturbed system of system \eqref{P2} with the form
\begin{equation}\begin{split}\label{PP2}
&\dot{x}=f(y,\varepsilon)=y\prod_{i=1}^k((y-\alpha_i)^2+\varepsilon),\\
&\dot{y}=-g(x,\varepsilon)=-x\prod_{i=1}^k((x-\beta_i)^2+\varepsilon),
\end{split}\end{equation}
where $0<\varepsilon\ll1$ is a real parameter.

Obviously, system \eqref{PP2} is also a Hamiltonian system with Hamiltonian function
\begin{equation}\label{H4}
H(x,y,\varepsilon)=F(y,\varepsilon)+G(x,\varepsilon),\ \mbox{where}\ F(y,\varepsilon)=\int_{0}^y f(s,\varepsilon)\ \mathrm{d}s,\ G(x,\varepsilon)=\int_{0}^x g(s,\varepsilon)\ \mathrm{d}s.
\end{equation}
It is easy to verify that system \eqref{PP2} has a unique equilibrium at $(0,0)$, which is a global center.
Thus, the period annulus around the center of system \eqref{PP2} is defined by the Hamiltonian function $H(x,y,\varepsilon)=h,\ h\in(0,+\infty)$.
The period function $T(h,\varepsilon)$ corresponding to this period annulus is given by
\begin{equation}
T(h,\varepsilon)=\oint_{\Gamma_{h,\varepsilon}}\frac{1}{f(y,\varepsilon)}\ \mathrm{d}x,\qquad h\in(0,+\infty),
\end{equation}
where $\Gamma_{h,\varepsilon}$ is the periodic orbit defined by $H(x,y,\varepsilon)=h$.

We obtain a result on the lower bound for the number of critical periods for
polynomial systems as follows.
\begin{theorem}\label{Th2}
For the period annulus of system \eqref{PP2} with $\varepsilon$ sufficiently small, the corresponding period function $T(h,\varepsilon)$ has at least $2k^2+4k-1$ critical points on $(0,+\infty)$.
\end{theorem}
\begin{proof}
Firstly, from Corollary \ref{Pcoro}, we can suppose that the period function $T(h)$ has a critical point at $h=\bar h_s\in(h_s,h_{s+1})$.
Note that each $T(\bar h_s)$ is finite, thus we let $M=\max\{T(0)+1,\ T(\bar h_s)+1,\ s=1,2\cdots,(k+1)^2-2\}$. Since $T(\bar h_s,\varepsilon)$ is continuous with respect to $\varepsilon$ and $T(\bar h_s,0)=T(\bar h_s)$, there exists $\bar\varepsilon_s>0$ such that $T(\bar h_s,\varepsilon)<M$ for $0<\varepsilon<\bar \varepsilon_s$. Similarly, there exist $\bar\varepsilon_0>0$ and $\bar\varepsilon_{(k+1)^2-1}>0$ such that $\lim_{h\rightarrow0^+}T(h,\varepsilon)<M$ holds for $0<\varepsilon<\bar \varepsilon_0$ and
$\lim_{h\rightarrow+\infty}T(h,\varepsilon)<M$ holds for $0<\varepsilon<\bar\varepsilon_{(k+1)^2-1}$.

Secondly, by Proposition \ref{GThi}, we know that $\lim_{h\rightarrow h^-_s}T(h)=+\infty$, thus it is not difficult to verify that
for $M$ defined above, there exists $\varepsilon_s>0$ such that
$T(h_s,\varepsilon)>M$ for $0<\varepsilon<\varepsilon_s$.

Let $\varepsilon_0=\min\{\varepsilon_1,\varepsilon_2,\cdots,\varepsilon_{(k+1)^2-1}, \bar\varepsilon_0,\bar\varepsilon_1,\cdots,\bar\varepsilon_{{(k+1)^2-1}}\}$.
Applying Intermediate Value Theorem and  Rolle's theorem, we can show that the analytic function $T(h,\varepsilon)$ has at least $2k^2+4k-1$ critical points on $(0,+\infty)$.
Theorem \ref{Th2} is finished.
\end{proof}

\begin{theorem}\label{Th2e}
Consider polynomial system of degree $2k$ with the form
\begin{equation}\begin{split}\label{PPe}
&\dot{x}=y\prod_{i=1}^{k-1}((y-\alpha_i)^2+\varepsilon),\\
&\dot{y}=-x\prod_{i=1}^{k-1}((x-\beta_i)^2+\varepsilon)(\beta_k-x),
\end{split}\end{equation}
where $0<\varepsilon \ll1$ is a small parameter, and $\alpha_i$'s and $\beta_i$'s are real constants, satisfying
(i) $H(\beta_i,\alpha_j),\ i,j=0,1,\cdots,k-1$ differs from each others, with $\alpha_0=\beta_0=0$, and
(ii) $|H(\beta_k,0)|>\max\{|H(\beta_i,\alpha_j)|,\ i,j=0,1,\cdots,k-1\}$.
Here the function $H(x,y)$ is a first integral of the unperturbed system \eqref{PPe}$|_{\varepsilon=0}$.
Then the period function of system \eqref{PPe} corresponding to its period annulus has at least $2k^2-2=n^2/2-2$ critical points on $(0,\bar h_k)$ or $(\bar h_k,0)$ depending on $\beta_k>0$ or $\beta_k<0$, where
$\bar h_k=H(\beta_k,0,\varepsilon)$ with $H(x,y,\varepsilon)$ being a Hamiltonian first integral of system \eqref{PPe}.
\end{theorem}
\begin{proof}
We only consider the case $\beta_k>0$ and the case $\beta_k<0$ can be proved similarly. The difference in both cases is that the origin is a center or a saddle.

It is easy to verify that the unperturbed system \eqref{PPe}$|_{\varepsilon=0}$ has an elementary center at $(0,0)$, $2k-2$ cusps at $(\beta_i,0)$ and $(0,\alpha_i),\ i=1,2,\cdots,k-1$, a saddle at $(\beta_k,0)$ and
some degenerate singularities at $(\beta_i,\alpha_j),\ i,j=1,2\cdots,k-1$. Let
$$h_{s}=\{ \mbox{the $(s+1)$-th smallest element of $H(\beta_i,\alpha_j)$ in condition (i)}\},s=0,1,\cdots,k^2-1,$$
then we have similar conclusions as Proposition \ref{GThi} and Corollary \ref{Pcoro}. The perturbed system \eqref{PPe} has a center at $(0,0)$ and the period
annulus around this center is defined by $H(x,y,\varepsilon)=h, h\in(0,\bar h_k)$. Notice that the separatrix polycycle surrounding the period annulus is a saddle loop defined by $H(x,y,\varepsilon)=\bar h_k$.
Thus the period function $T(h,\varepsilon)$ corresponding to the period annulus of system \eqref{PPe} satisfies $\lim_{h\rightarrow \bar h_k^-} T(h,\varepsilon)=+\infty$. Using condition (ii), Theorem \ref{Th2e} follows from a similar proof as Theorem \ref{Th2}.
\end{proof}

To finish the proof, we need to show that there exist polynomial systems satisfying the conditions in Theorem \ref{Th2} and Theorem \ref{Th2e} respectively. When $n=2k+1$, the following theorem
shows that almost all the systems having the form  (\ref{P2}) satisfy the condition in Theorem \ref{Th2}, i.e., the hypothesis $(\mathrm{H})$.
\begin{theorem}\label{thH}
Let $(\mathbf{\alpha_m},\mathbf{\beta_k})=(\alpha_1,\alpha_2,\cdots,\alpha_m,\beta_1,\beta_2,\cdots,\beta_k)\in\mathbb{R}^{m+k}$, and
\begin{eqnarray*}
V_1=\left\{(\mathbf{\alpha_k},\mathbf{\beta_k})\ \left|
\left.\begin{array}{l}
H(\beta_i,\alpha_j), i,j=0,1,\cdots,k\ \mbox{differ from each others}
\end{array} \right.\right. \right\},
\end{eqnarray*}
where $\alpha_0=\beta_0=0$, and the function $H(x,y)$ here is the first integral of
system \eqref{PP2}$|_{\varepsilon=0}$. Then the set $V_1$, which defines the parameter spaces of systems \eqref{PP2},
is an open dense set.
\end{theorem}
\begin{proof}
Obviously, the set $V_1$ is dense in $\mathbb{R}^{2k}$, since $V_1$ is the domain obtained by removing
finite real hypersurfaces from $\mathbb{R}^{2k}$.

Moreover, $V_1$ is an open set of $\mathbb{R}^{2k}$. Denote $H(x,y)=H(x,y,\mathbf{\alpha_k},\mathbf{\beta_k})$.
Take a point $(\tilde{\mathbf{\alpha_k}},\tilde{\mathbf{\beta_k}})\in V_1$,
it follows that $H(\tilde\beta_i,\tilde\alpha_j,\tilde{\mathbf{\alpha_k}},\tilde{\mathbf{\beta_k}}), i,j=0,1,\cdots,k$ differ from each others with $\tilde\alpha_0=\tilde\beta_0=0$.
Thus, we rearrange them as the hyperpiesis ($\mathrm{H}'$) and let
$h_{\min}=\min\{(h_{s+1}-h_s)/2, s=0,1,\cdots,(k+1)^2-2\}$. By the continuity of $H(x,y,\tilde{\mathbf{\alpha_k}},\tilde{\mathbf{\beta_k}})$ with respect to the parameters, there exist $\delta_s>0, s=1,2,\cdots,(k+1)^2-1$, such that $|H(\beta_i,\alpha_j,\mathbf{\alpha_k},\mathbf{\beta_k})-H(\tilde\beta_i,\tilde\alpha_j,\tilde{\mathbf{\alpha_k}},\tilde{\mathbf{\beta_k}})|<h_{\min}$
hold for $\|(\mathbf{\alpha_k},\mathbf{\beta_k})-(\tilde{\mathbf{\alpha_k}},\tilde{\mathbf{\beta_k}})\|<\delta_s$, where $i,j=0,1,\cdots,k$. Let
$\delta_{\min}=\min\{\delta_s, s=1,2,\cdots,(k+1)^2-1\}$. Then we have the $\delta_{\min}$ neighbourhood of $(\tilde{\mathbf{\alpha_k}},\tilde{\mathbf{\beta_k}})$ is also contained in $V_1$, and hence $V_1$ is open.
\end{proof}

At lat, we give two concrete examples.

{\bf Example 1.} Consider the polynomial system of degree $2k+1$
\begin{equation}\begin{split}\label{e2}
&\dot{x}=y\prod_{i=1}^k((y-i)^2+\varepsilon),\\
&\dot{y}=-x\prod_{i=1}^k((x-\mathrm{e}i)^2+\varepsilon),
\end{split}\end{equation}
where $0<\varepsilon\ll1$ is a real parameter and $\mathrm{e}$ is the natural constant.

Obviously, the unperturbed system \eqref{e2}$|_{\varepsilon=0}$ has a first integral
\begin{equation}\begin{split}\label{H2}
H(x,y)&=\int_{0}^y s\prod_{i=1}^k(s-i)^2\ \mathrm{d}s+\int_{0}^x s\prod_{i=1}^k(s-\mathrm{e}i)^2\ \mathrm{d}s,
\end{split}\end{equation}
and system \eqref{e2}$|_{\varepsilon=0}$ has $(k+1)^2$ singularities $(\mathrm{e}i,j),\ i,j=0,1,\cdots,k$.

First of all, we have an important result
to show that the hypothesis $(\mathrm{H})$ holds for the first integral \eqref{H2}.
\begin{proposition}\label{Hi}
The Hamiltonian values $H(\mathrm{e}i,j), i,j=0,1,\cdots,k$ differ from each others.
\end{proposition}
\begin{proof}
Firstly, by the transformation $s=\mathrm{e}u$,
\begin{equation*}\begin{split}
H(x,y)
&=\int_{0}^y s\prod_{i=1}^k(s-i)^2\ \mathrm{d}s+\mathrm{e}^2\int_{0}^{x/\mathrm{e}} u\prod_{i=1}^k(\mathrm{e}u-\mathrm{e}i)^2\ \mathrm{d}u\\
&=F(y)+\mathrm{e}^{2k+2}F(x/\mathrm{e}), \qquad \mbox{where}\ F(z)=\int_{0}^z s\prod_{i=1}^k(s-i)^2\ \mathrm{d}s.
\end{split}\end{equation*}
Hence the Hamiltonian values
$$H(\mathrm{e}i,j)=F(j)+\mathrm{e}^{2k+2}F(i).$$

Secondly, we derive two properties of the function $F(z)$. \\
($\mathrm a$) $F(i)$ and $F(j)$ are rational numbers, since $F(z)$ is a polynomial of $z$ with the rational coefficients.\\
($\mathrm{b}$) $F(i)<F(j)$, if $0\leq i<j$. It follows from $$F'(z)=z\prod_{i=1}^k(z-i)^2\geq0,\quad \mbox{for}\ z\geq0$$
that $F(z)$ increases on the interval $[0,+\infty)$.

For two different singularities $(\mathrm{e}i_1,j_1)$ and $(\mathrm{e}i_2,j_2)$, if $i_1=i_2$ and $j_1\neq j_2$, then by the property ($\mathrm{b}$),  $H(\mathrm{e}i_1,j_1)\neq H(\mathrm{e}i_2,j_2)$.
If $i_1\neq i_2$, using these two properties and proof by contradiction, we can show that
if $(\mathrm{e}i_1,j_1)$ and $(\mathrm{e}i_2,j_2)$ such that $H(\mathrm{e}i_1,j_1)=H(\mathrm{e}i_2,j_2)$,
then
$$\frac{F(j_1)-F(j_2)}{F(i_2)-F(i_1)}=\mathrm{e}^{2k+2},$$
which leads to a contradiction since the left side is a rational number and the right side is an irrational number. Therefore, the proposition follows.
\end{proof}

By Proposition \ref{Hi}, we let
\begin{equation*}
h_{s}=\{ \mbox{the $(s\!+\!1)$-th smallest element of $H(\mathrm{e}i,j)$}, \ i,j=0,1,\cdots,k\},\ s=0,1,\cdots,(k+1)^2-1.
\end{equation*}
It is easy to verify that system \eqref{e2} has a global center at the origin $(0,0)$, and has a first integral
\begin{equation*}
H(x,y,\varepsilon)=\int_{0}^y s\prod_{i=1}^k((s-i)^2+\varepsilon)\ \mathrm{d}s+\int_{0}^x s\prod_{i=1}^k((s-\mathrm{e}i)^2+\varepsilon)\ \mathrm{d}s.
\end{equation*}
The period annulus around the center of system \eqref{e2} is defined by the Hamiltonian function $H(x,y,\varepsilon)=h,\ h\in(0,+\infty)$.
Applying Theorem \ref{Th2}, we can deduce that for $\varepsilon$ sufficiently small,
the period function corresponding to the period annulus of system \eqref{e2} has at least $2k^2+4k-1$ critical points on the interval $(0,+\infty)$.

{\bf Example 2.} Consider the polynomial system of degree $2k$
\begin{equation}\begin{split}\label{e3}
&\dot{x}=y\prod_{i=1}^{k-1}((y-i)^2+\varepsilon),\\
&\dot{y}=-x\prod_{i=1}^{k-1}((x-\mathrm{e}i)^2+\varepsilon)(\mathrm{e}k^2-x),
\end{split}\end{equation}
where $0<\varepsilon\ll1$ is a real parameter and $\mathrm{e}$ is the natural constant.

Obviously, the unperturbed system \eqref{e3}$|_{\varepsilon=0}$ has a first integral
\begin{equation}\begin{split}\label{H2e}
H(x,y)&=\int_{0}^y s\prod_{i=1}^{k-1}(s-i)^2\ \mathrm{d}s+\int_{0}^x s\prod_{i=1}^{k-1}(s-\mathrm{e}i)^2(\mathrm{e}k^2-s)\ \mathrm{d}s,
\end{split}\end{equation}
and system \eqref{e3}$|_{\varepsilon=0}$ has $k(k+1)$ singularities $(\mathrm{e}i,j),\ i,j=0,1,\cdots,k-1$ and $(\mathrm{e}k^2,j),\ j=0,1,\cdots,k-1$,
where the singularity $(0,0)$ is a center and the singularity $(\mathrm{e}k^2,0)$ is a saddle.

In the following, we will verify that the conditions (i) and (ii) in Theorem \ref{Th2e} hold for the first integral \eqref{H2e}.
Similar as Proposition \ref{Hi}, it is easy to obtain that
\begin{proposition}\label{Hi1}
(i) The Hamiltonian values $H(\mathrm{e}i,j), i,j=0,1,\cdots,k-1$ differ from each others;
(ii) $H(\mathrm{e}k^2,0)>\max\{H(\mathrm{e}i,j),\ i,j=0,1,\cdots,k-1\}$.
\end{proposition}
\begin{proof}
Firstly, by the transformation $s=\mathrm{e}u$,
\begin{equation*}\begin{split}
H(x,y)
&=\int_{0}^y s\prod_{i=1}^{k-1}(s-i)^2\ \mathrm{d}s+\mathrm{e}^{2k+1}\int_{0}^{x/\mathrm{e}} u\prod_{i=1}^{k-1}(u-i)^2(k^2-u)\ \mathrm{d}u\\
&=F(y)+\mathrm{e}^{2k+1}G(x/\mathrm{e}),
\end{split}\end{equation*}
where
$$F(z)=\int_{0}^z s\prod_{i=1}^{k-1}(s-i)^2\ \mathrm{d}s,\qquad G(z)=\int_{0}^{z} s\prod_{i=1}^{k-1}(s-i)^2(k^2-s)\ \mathrm{d}s.$$
Hence the Hamiltonian values
$$H(\mathrm{e}i,j)=F(j)+\mathrm{e}^{2k+1}G(i).$$

Secondly, two properties of the functions $F(z)$ and $G(z)$ are derived. \\
($\mathrm a$) $G(i)$ and $F(j)$ are rational numbers, since $G(z)$ and $F(z)$ are polynomials of $z$ with the rational coefficients.\\
($\mathrm{b}$) $F(i)<F(j)$ and $G(i)<G(j)$, if $0\leq i<j$. It follows from
\begin{equation*}\begin{split}
&F'(z)=z\prod_{i=1}^{k-1}(z-i)^2\geq0,\quad \mbox{for}\ z\geq0,\\
&G'(z)=z\prod_{i=1}^{k-1}(z-i)^2(k^2-z)\geq0,\quad \mbox{for}\ 0\leq z\leq k^2,
\end{split}\end{equation*}
that $F(z)$ increases on the interval $[0,+\infty)$ and $G(z)$ increases on the interval $[0,k^2]$.

For two different singularities $(\mathrm{e}i_1,j_1)$ and $(\mathrm{e}i_2,j_2)$, if $i_1=i_2$ and $j_1\neq j_2$, then by the property ($\mathrm{b}$),  $H(\mathrm{e}i_1,j_1)\neq H(\mathrm{e}i_2,j_2)$.
If $i_1\neq i_2$, using these two properties and proof by contradiction, we can show that
if $(\mathrm{e}i_1,j_1)$ and $(\mathrm{e}i_2,j_2)$ such that $H(\mathrm{e}i_1,j_1)=H(\mathrm{e}i_2,j_2)$,
then
$$\frac{F(j_1)-F(j_2)}{G(i_2)-G(i_1)}=\mathrm{e}^{2k+1},$$
which leads to a contradiction since the left side is a rational number and the right side is an irrational number. Therefore, the conclusion (i) follows.

It follows from the property (b) that $$\max\{H(\mathrm{e}i,j),\ i,j=0,1,\cdots,k-1\}=H(\mathrm{e}(k-1),k-1).$$
Thus to finish conclusion (ii), we only need to prove $H(\mathrm{e}k^2,0)>H(\mathrm{e}(k-1),k-1)$. This is obvious for $k=1$, and for $k\geq2$, we have that
\begin{equation*}\begin{split}
H(\mathrm{e}k^2,0)=&e^{2k+1}\int_{0}^{k^2} s\prod_{i=1}^{k-1}(s-i)^2(k^2-s)\ \mathrm{d}s\\
=&e^{2k+1}G(k-1)+e^{2k+1}\int_{k-1}^{k^2} s\prod_{i=1}^{k-1}(s-i)^2(k^2-s)\ \mathrm{d}s\\
>&e^{2k+1}G(k-1)+e^{2k+1}\int_{k}^{2k-1} s\prod_{i=1}^{k-1}(s-i)^2(k^2-s)\ \mathrm{d}s\\
=&e^{2k+1}G(k-1)+e^{2k+1}\int_{0}^{k-1} (s+k)\prod_{i=1}^{k-1}(s+k-i)^2(k^2-s-k)\ \mathrm{d}s\\
>&e^{2k+1}G(k-1)+e^{2k+1}\int_{0}^{k-1} s\prod_{i=1}^{k-1}(s+k-i)^2\ \mathrm{d}s\\
>&e^{2k+1}G(k-1)+e^{2k+1}F(k-1)>H(\mathrm{e}(k-1),k-1).
\end{split}\end{equation*}
\end{proof}

It is easy to verify that system \eqref{e3} also has a center at the origin $(0,0)$ and a saddle at $(\mathrm{e}k^2,0)$. By Proposition \ref{Hi1} and Theorem \ref{Th2e}, we can deduce that for $\varepsilon$ sufficiently small,
the period function corresponding to the period annulus of system \eqref{e3} has at least $2k^2-2$ critical points on the interval $(0,\bar h_k)$ with $\bar h_k=H(\mathrm{e}k^2,0,\varepsilon)$, where
\begin{equation*}
H(x,y,\varepsilon)=\int_{0}^y s\prod_{i=1}^{k-1}((s-i)^2+\varepsilon)\ \mathrm{d}s+\int_{0}^x s\prod_{i=1}^{k-1}((s-\mathrm{e}i)^2+\varepsilon)(\mathrm{e}k^2-s)\ \mathrm{d}s.
\end{equation*}

\end{document}